\documentclass[12pt]{amsart}

\usepackage{amsfonts,amssymb,amsmath,textcomp,amsthm}
\usepackage{color}
\usepackage{hyperref}
\usepackage{amstext} % for \text macro
\usepackage{array}   % for \newcolumntype macro
\newcolumntype{C}{>{$}c<{$}} % math-mode version of "c" column type

\usepackage{amsmath,amssymb}
\usepackage{enumitem}  % para personalizar listas numeradas

\usepackage{enumitem}
\newtheorem{theorem}{Theorem}[section]
\theoremstyle{plain}

\newtheorem{condition}[theorem]{Condition}

\newtheorem{corollary}[theorem]{Corollary}

\newtheorem{lemma}[theorem]{Lemma}

\newtheorem{problem}[theorem]{Problem}
\newtheorem{proposition}[theorem]{Proposition}

\numberwithin{equation}{section}

\theoremstyle{definition}
\newtheorem{definition}[theorem]{Definition}

\newtheorem{remark}[theorem]{Remark}

\newcommand{\aut}{\mathrm{Aut}\,}	            % automorphism group
\newcommand{\half}{\mathrm{Half}\,}             % half-automorphism group
\newcommand{\out}{\mathrm{Out}\,}               % Outer autommorphisms
\def\subsp#1{\langle #1\rangle}             % subspace

\newcommand{\cl}{\mathrm{CodeLoop}}	            % automorphism group

\begin{document}
	\title{The Half-automorphism Group of Code Loops}

	\author[de Barros]{Dylene Agda Souza de Barros}
	\author[Grishkov]{Alexandre Grichkov}
	\author[Miguel Pires]{Rosemary Miguel Pires}
	\author[Rasskazova]{Marina Rasskazova}
	\address[de Barros]{Instituto de Matem\'atica e Estat\'istica, Universidade Federal de Uberl\^{a}ndia, Av. Jo\~{a}o Naves de \'Avila, 2121, Campus Santa M\^onica, Uberl\^{a}ndia, MG, Brazil, CEP 38408-100}
	\email[de Barros]{dylene@ufu.br}
	\address[Grishkov]{Instituto de Matem\'atica e Estat\'istica, Universidade de S\~ao Paulo, Rua do Mat\~ao, 1010 Butant\~a, S\~ao Paulo, SP, Brazil, CEP 05508-090  and Sobolev Institute of Mathematics, Omsk, Russia}
	\email[Grishkov]{shuragri@gmail.com}
	\address[Miguel Pires]{Instituto de Ci\^encias Exatas, Universidade Federal Fluminense, Rua Desembargador Ellis Hermydio Figueira, Aterrado, Volta Redonda, RJ, Brazil, CEP 27213-145 } 
	\email[Pires]{rosemarypires@id.uff.br}
	\address[Rasskazova]{Centro de Matem\'atica, Computa\c c\~ao  e Cogni\c c\~ao, Universidade Federal do ABC, Avenida dos Estados, 5001, Bairro Bangu, Santo Andr\'e, SP, Brasil,
		CEP 09280-560; and Siberian State Automobile and Highway University, Prospekt Mira, 5, Omsk, Omsk Oblast, Rússia, 644080}
	\email[Rasskazova]{marina.rasskazova@ufabc.edu.br}
	
	\begin{abstract}
		For any code loop $L$, we prove that the half-automorphism group of $L$ is the product of the automorphism group of $L$ by an elementary abelian $2-$group consisting of all half-automorphisms that acts as the identity on a fixed basis. Also, we prove that elementary mappings only can be a half-automorphism on code loops of rank at most $3$.
	\end{abstract}


	\maketitle
	\section{Introduction}\label{Sc:Intro}
	
	A \emph{loop} is a set $L$ with a binary operation $\cdot$ and a neutral element $1\in L$ such that for every $a$, $b\in L$ the equations $ax=b$ and $ya=b$ have unique solutions $x$, $y\in L$, respectively.
	Let $(L,\ast)$ and $(L',\cdot)$ be loops. A bijection $f:L\to L'$ is a \emph{half-isomorphism} if $f(x\ast y)\in\{f(x)\cdot f(y),f(y)\cdot f(x)\}$, for every $x$, $y\in L$. A \emph{half-automorphism} is defined as expected. We say that a half-isomorphism (half-automorphism) is \emph{trivial} if it is either an isomorphism (automorphism) or an anti-isomorphism (anti-automorphism). A \emph{code loop} is a finite Moufang loop $L$ with a unique (central) nontrivial square (associator, commutator).
	
	In 1957, W.R. Scott proved every half-isomorphism between two groups is trivial \cite{Scott}. He also gave an example of a loop of order $8$ that has a nontrivial half-automorphism showing that his result does not generalize to all loops. However, we know now that Scott's result can be extended for some special cases of loops: Moufang loops in which the squaring map in the factor loop over their nucleus is surjective, automorphic loops of odd order \cite{KSV,GA21}.
	
	On the other hand, investigations on loops that have nontrivial half-automorphisms were also made. Gagola and Giuliani \cite{GG2} established conditions for the existence of nontrivial half-automorphisms for certain Moufang loops of even order, and dos Anjos described when a Chein loop has nontrivial half-automorphims, \cite{G21}. Another type of investigation in this topic is to describe the half-automorphism group of a loop that has nontrivial half-automorphisms \cite{BG, BP, GA19, GA20}.
	
	In this work, we approach the problem of describing the half-automorphism group of a code loop $L$ by considering the group $H_\mathcal{B}(L)$ of all half-automorphisms that act as identity on a fixed basis $\mathcal{B}$ of $L$. In Section \ref{Sc:HA} we prove that every half-automorphism of a code loop $L$ is a product of an automorphism and an element of $H_\mathcal{B}(L)$. As a consequence of this, we proved that a non-associative code loop $L$ has only trivial half-automorphisms if and only if $|H_\mathcal{B}(L)|=2$ for some basis $\mathcal{B}$, solving Problem 6.4 of \cite{BP}. In Section \ref{Sc:Rank3and4}, we apply the results of Section \ref{Sc:HA} for all code loops of rank $3$ and $4$. In Section \ref{Sc:elementary-map} we prove that in code loops of rank higher than $3$, for any $x$, the elementary mapping $\tau_x$ is not a half-automorphism, solving Problem 6.5 of \cite{BP}.

	%%%%%%%%%%%%%%%%
	\section{Preliminaires}\label{Sc:Preliminaires}
	
	A \emph{Moufang loop} is a loop $L$, in which $(x\cdot zx)y = x(z\cdot xy)$ for every $x$, $y$ and $z\in L$. Moufang loops are diassociative, IP loops and, whenever three elements associate, they generate a subgroup. Let $L$ be a loop. Denote by $L^2$, $Z(L)$, $A(L)$ and $[L,L]$ the set of all \emph{squares}, the \emph{center}, 
	the \emph{associator subloop}, and the \emph{commutator subloop} of $L$, respectively. For further results on Moufang loops and loops, see \cite{P90}.
	\subsection{Code loops} A code loop is a Moufang loop $L$ with a unique nontrivial square. Chein and Goodaire proved that for such loops, $A(L)=[L,L]=L^2\leq Z(L)$. We denote the unique nontrivial square of a code loop by $-1$ and for a nonassociative code loop we have $A(L)=[L,L]=L^2=\{1,-1\}$ and we write $(-1)x$ as $-x$. It is known that for every $x,y,z,w$ in a code loop $L$, $[xy,z]=[x,y][y,z](x,y,z)$, $(xy)^2=x^2y^2[x,y]$ and $(wx,y,z)=(w,y,z)(x,y,z)$ \cite{BP}. One can read \cite{Rose,Griess,CG} for details about code loops.
	
	A minimal set of generators for a code loop $L$ is called a \emph{basis} of $L$, and we say that $L$ has \emph{rank} $n$, if it has a basis with $n$ elements. If $\mathcal{B}=\{a_1,a_2,a_3\ldots,a_n\}$ is a basis of a code loop $L$, then every element of $L$ is of the form $\pm (\ldots(a_1a_2)a_3)\ldots a_{n-1})a_n$. In addition, the vector $\overline{\lambda}=(\lambda_1,\ldots,\lambda_n,\lambda_{1,2},\ldots,\lambda_{n-1,n},\lambda_{1,2,3},\ldots,\lambda_{n-2,n-1,n})$ where $\lambda_i$, $\lambda_{i,j}$, $\lambda_{i,j,k}\in\{0,1\}$, with $i<j<k$, is given by $a_i^2=(-1)^{\lambda_i}$, $[a_i,a_j]=(-1)^{\lambda_{i,j}}$ and $(a_i,a_j,a_k)=(-1)^{\lambda_{i,j,k}}$, respectively, is called \emph{characteristic vector} of $L$, related to $\mathcal{B}$. In this paper, the choice of a basis of a code loop $L$ and its related characteristic vector plays an important role in the description of the half-automorphism group of $L$ in terms of the automorphism group of $L$. In \cite{Rose,GP18} a classification of code loops of rank $3$ and $4$ according to the characteristic vector is presented.
	
	\begin{theorem}[\cite{GP18}, Theorem 3.1]\label{Th:G-Rose-posto3}
		Consider $C_{1}^{3},...,C_{5}^{3}$ the code loops with the following characteristic vectors:
		\begin{eqnarray*}
			\begin{array}{lllll}
				\lambda(C_{1}^{3})=(1,1,1,1,1,1), &&\lambda(C_{2}^{3})=(0,0,0,0,0,0),&&\lambda(C_{3}^{3})=(0,0,0,1,1,1),\\
				\lambda(C_{4}^{3})=(1,1,0,0,0,0),&&\lambda(C_{5}^{3})=(1,0,0,0,0,0).&&\\
			\end{array}
		\end{eqnarray*}
		Then any two loops from the list $\left\{C_{1}^{3},...,C_{5}^{3}\right\}$  are not isomorphic and every nonassociative
		code loop of rank $3$ is isomorphic to one of this list.
	\end{theorem}
	
	In the case where $L$ is a code loop of rank $4$, Grichkov and Miguel Pires proved that $L$ possesses a basis $\mathcal{B}=\{a_1,a_2,a_3,a_4\}$, where $a_4$ is an element which is in the nucleus of $L$. Thus, the last four coordinates of the characteristic vector related to $\mathcal{B}$ are $1000$, and these coordinates were omitted in the classification theorem.
	
	\begin{theorem}[\cite{GP18}, Theorem 3.5]\label{Th:G-Rose-posto4}
		Consider $C_{1}^{4},...,C_{16}^{4}$ the code loops with the following characteristic vectors. All nonassociative code loops of rank $4$ are isomorphic to one from the list. Moreover, none of those loops is isomorphic to each other.
		\begin{table}[!hhh]
			\centering
			{\begin{tabular}{llll|lllll|l}
					&  &  & $L$ & $\lambda(L)$ &  &  &  & $L$ & $\lambda(L)$ \\
					\cline{4-5}\cline{9-10}
					&  &  & $C^{4}_{1}$ & $(1110110100)$ &  &  &  & $C^{4}_{9}$ & $(0100001000)$ \\
					&  &  & $C^{4}_{2}$ & $(0000000000)$ &  &  &  & $C^{4}_{10}$ & $(0001111000)$ \\
					&  &  & $C^{4}_{3}$ & $(0000110100)$ &  &  &  & $C^{4}_{11}$ & $(0001001000)$ \\
					&  &  & $C^{4}_{4}$ & $(0010100000)$ &  &  &  & $C^{4}_{12}$ & $(0000001100)$ \\
					&  &  & $C^{4}_{5}$ & $(0000010100)$ &  &  &  & $C^{4}_{13}$ & $(0110111100)$ \\
					&  &  & $C^{4}_{6}$ & $(1111110100)$ &  &  &  & $C^{4}_{14}$ & $(0001001100)$ \\
					&  &  & $C^{4}_{7}$ & $(0001000000)$ &  &  &  & $C^{4}_{15}$ & $(1001001100)$ \\
					&  &  & $C^{4}_{8}$ & $(0000001000)$ &  &  &  & $C^{4}_{16}$ & $(0001111100)$ \\
			\end{tabular}}
			\label{tabvc4}
		\end{table}
	\end{theorem}
	\newpage
	\subsection{Half-isomorphism of code loops}
	Recall that a \emph{half-automorphism} of a loop $L$ is a bijection $f:L\to L$ such that $f(xy)\in\{f(x)f(y),f(y)f(x)\}$, for all $x,y\in L$. A half-automorphism of a loop $L$ is called \emph{trivial} if it is either an automorphism or an anti-automorphism. One can search for properties of half-isomorphisms of loops in \cite{GA19, GA21}.
	
	A half-isomorphism $f:L\to L'$ between two loops is called \emph{special} if its inverse $f^{-1}:L'\to L$ is also a half-isomorphism. Every half-isomorphism between two diassociative loops is special, (Corollary 3.6, \cite{G21-2}). Now, we show that if there is a half-isomorphism between two code loops, then they are isomorphic.
	
	\begin{lemma}\label{Lm:Center-P0}
		Suppose that $P_0=\langle a,b,c\rangle$ is a $3-$generated subgroup of a code loop $P$. Then either $P_0$ is abelian or $P_0$ has center of order $4$.
	\end{lemma}
	\begin{proof}
		If $P_0$ is not abelian, let us assume that $ab\neq ba = -ab$. Then $\langle a,b\rangle$ is either the dihedral group $D_8$ or the quaternion group $Q_8$. Observe that $\langle a,b,c\rangle = \langle a, b, bc\rangle$ and if $[a,c]=-1$, then $[a,bc]=1$; therefore we may assume $[a,c]=1$. Similarly, $\langle a,b,c\rangle = \langle ac, b, c\rangle$ and if $[b,c]=-1$, then $[b,ac]=[ac,c]=1$; therefore, we also may assume $[b,c]=1$. Then $\langle c,-1\rangle\leq Z(P_0)$ and a simple calculation shows that $Z(P_0)=\langle c,-1\rangle$. Then, if $c^2=1$, then $Z(P_0)\cong\mathbb{Z}_2\times\mathbb{Z}_2$ and if $c^2=-1$, then $Z(P_0)\cong\mathbb{Z}_4$.
	\end{proof}
	
	Given a loop $L$, $C(L)=\{c\in L; cx=xc,\,\mbox{for all}\, x\in L\}$ is called \emph{commutant} of $L$.
	
	\begin{lemma}\label{Lm:commutant}
		Let $P$ and $Q$ be diassociative loops. If $f:P\to Q$ is a half-isomorphism, then $f(C(P))=C(Q)$.
	\end{lemma}
	\begin{proof}
		By Lemma 2.1 of \cite{KSV}, and since $f^{-1}$ is also an half-isomporphism, we have if $xy=yx$ if and only if $f(x)f(y) = f(y)f(x)$ and the result follows.
	\end{proof}
	\begin{proposition}\label{Pr:code-loops-isomorphic}
		Let $P$ and $Q$ be finite code loops and let $f:P\to Q$ be a half-isomorphism. Then $P$ is isomorphic to $Q$.
	\end{proposition}
	\begin{proof}
		We shall prove that $P$ and $Q$ have the same characteristic vector. If $x^2=1_P$ for all $x\in P$, then $P$ is an abelian group and, since $x$ and $f(x)$ have the same order, $Q$ is also an abelian group, and then $P\cong Q$. If $x^2=-1_P$, for some $x\in P$, we get $f(-1_P)=f(x^2)=f(x)^2=-1_Q$, because code loops have a unique nontrivial square. For $x,y\in P$, we get $f([x,y])=[f(x),f(y)]$, because for (special) half-isomorphism between code loops, $xy=yx$ if and only if $f(x)f(y)=f(y)f(x)$. Let $x$, $y$ and $z\in P$. We wish to prove that $f(x,y,z)=(f(x),f(y),f(z))$. If both $P_0=\langle x,y,z\rangle$ and $Q_0=\langle f(x),f(y),f(z)\rangle$ are groups, then $f(x,y,z)=f(1_P)=1_Q=(f(x),f(y),f(z))$. Similarly, if both $P_0$ and $Q_0$ are nonassociative, then $f(x,y,z)=f(-1_P)=-1_Q=(f(x),f(y),f(z))$. Suppose that $P_0$ is a group and $Q_0$ is nonassociative. Recall that, since $Q_0$ is a code loop of rank $3$, $Z(Q_0)=\{1_Q,-1_Q\}$. By Lemmas \ref{Lm:Center-P0} and \ref{Lm:commutant}, there exists $c\in Z(P_0)$, with $c\neq\pm 1_P$ and $f(c)\in C(Q_0)$. Then for every $u$ and $v\in Q_0$, we get $1_Q=[f(c),uv]=[f(c),u][f(c),v](f(c),u,v)$, then $f(c)\in Z(Q_0)$, which is a contradiction. Then $f(x,y,z)=(f(x),f(y),f(z))$, for all $x$, $y$ and $z\in P$. Finally, if $\{a_1,\ldots,a_n\}$ is a basis for $P$, then $\{f(a_1),\ldots,f(a_n)\}$ is a basis for $Q$ and they have the same related characteristic vector. Hence, $P\cong Q$.
	\end{proof}

	It is known that every half-automorphism of a finite loop is special and if $xy=yx$, then $f(xy)=f(x)f(y)$. Also, if $L$ is a code loop, every half-automorphism of $L$ preserves the unique nontrivial square of $L$. More precisely:
	
	\begin{lemma}[\cite{BP}, Lemma 2.9]\label{Lm:auto-ger}
		Let $L$ be a code loop and let $f$ be a half-automorphism of $L$.
		\begin{itemize}
			\item[(i)] $f(-1)=-1$.
			\item[(ii)] $f((x,y,z))=(f(x),f(y),f(z))$ for every $x,y,z\in L$.
			\item[(iii)] $[f(-x),f(y)]=[-f(x),f(y)]=[f(x),-f(y)]=[f(x),f(y)]=[x,y]$ for every $x,y\in L$.
		\end{itemize}
	\end{lemma}
	\begin{corollary}\label{Cr:same-char-vec}
		Let $L$ be a code loop with fixed basis $\mathcal{B}=\{a_1,a_2,\ldots,a_n\}$ and let $f$ be a half-automorphism of $L$. Then $\{f(a_1),f(a_2),\ldots,f(a_n)\}$ is a basis of $L$ with the same related characteristic vector.
	\end{corollary}
	\begin{proof}
		If $x=\pm(\ldots(a_1a_2)a_3)\ldots a_{n-1})a_n$ is an element of $L$ then, since $f(-1)=-1\in A(L)=[L,L]\leq Z(L)$, we get $f(x)=\pm(\ldots(f(a_1)f(a_2))f(a_3))\ldots f(a_{n-1}))f(a_n)$. Thus, $f(\mathcal{B})=\{f(a_1),f(a_2),\ldots,f(a_n)\}$ is a basis for $L$. Also, by Lemma \ref{Lm:auto-ger}, the characteristic vectors related to $\mathcal{B}$ and $f(\mathcal{B})$ are the same.
	\end{proof}
	
	%%%%%%%%%%%%%%%%
	\section{The structure of $\half(L)$}\label{Sc:HA}
	
	For a loop $L$ we denote by $\half(L)$ and by $\aut(L)$ the groups of half-automorphisms and automorphisms of $L$, respectively. Let $L$ be a code loop of rank $n$ with fixed basis $\mathcal{B}=\{a_1,\ldots,a_n\}$ and let $P_n$ be the set of subsets of $I_n=\{1,\ldots,n\}$. We define $a_\emptyset=1$, and for an ordered subset $\sigma=\{i_1,\ldots,i_k\}\in P_n$, we define $a_\sigma\in L$ by $a_\sigma=a_\mu\cdot a_{i_k}$, where $\mu=\{1,\ldots,i_{k-1}\}$. For example, $a_1a_4\cdot a_5$ is written as $a_\sigma$, where $\sigma=\{1,4,5\}$ and, for $\mu=\{1,3,4,5\}$, $a_\mu=(a_1a_3\cdot a_4)a_5$. Since $L^2=[L,L]=(L,L,L)=\{1,-1\}\leq Z(L)$, we have $L=\{\pm\, a_\sigma\,:\,\sigma\in P_n\}$. If $a_\sigma$ and $a_\mu\in L$, there is $\lambda(\sigma,\mu)\in\{1,-1\}$ such that
	\begin{equation}\label{product}
		a_\sigma\cdot a_\mu = \lambda(\sigma,\mu)a_{\sigma\Delta\mu},
	\end{equation}
	in which $\sigma\Delta\mu$ denotes the symmetric difference of $\sigma$ and $\mu$. Denote $\{\sigma,\mu\}\in\{1,-1\}$ the commutator $[a_\sigma,a_\mu]$.
	
	Define $H_\mathcal{B}(L)$ by the set of all half-automorphisms, $\varphi$ of $L$ such that $\varphi(a_\sigma)=\varepsilon_\sigma a_\sigma$, where $\varepsilon_\sigma\in\{1,-1\}$, $\epsilon_{\{j\}}=1$, for all $j\in I_n$. Actually, $H_\mathcal{B}(L)$ is the set of all half-automorphisms of $L$ that act as identity on the fixed basis $\mathcal{B}$ and it is easy to see that $H_\mathcal{B}(L)$ is a subgroup of $\half(L)$. On the other hand, let $R_\mathcal{B}$ be the set of all functions $\Psi:P_n\to\{-1,1\}$ satisfying $\Psi(\emptyset) = 1$ $, \Psi(\{j\}) = 1$, for every $j\in I_n$ and $\Psi(\sigma\Delta\mu)=\Psi(\sigma)\Psi(\mu)$, whenever $\{\sigma,\mu\}=1$. For any $\Psi\in R$, define $\hat{\Psi}:L\to L$ by $\hat{\Psi}(a_\sigma)=\Psi(\sigma)a_\sigma$ and $\hat{\Psi}(-a_\sigma)=-\Psi(\sigma)a_\sigma$.
	\begin{proposition}\label{Pr:RBHB}
		If $\Psi$ is an element of $R_\mathcal{B}$, then $\hat{\Psi}$ is a half-automorphism in $H_\mathcal{B}(L)$. Conversely, if $\varphi$ is a half-automorphism in $H_\mathcal{B}(L)$, then there exists a function $\Psi$ in $R_\mathcal{B}$ such that $\varphi = \hat{\Psi}$.
	\end{proposition}
	\begin{proof}
		
		Suppose $\Psi\in R_\mathcal{B}$ and let us prove that $\hat{\Psi}$ is a half-automorphism in $H_\mathcal{B}(L)$. First, by the definition of $\hat{\Psi}$, we get $\hat\Psi(1)=1$ and $\hat\Psi(a_j)=a_j$ for every $j\in I_n$. Let $a_\sigma$ and $a_\mu$ be elements of $L$. Observe that $\hat{\Psi}(a_\sigma a_\mu)=\hat{\Psi}(\lambda(\sigma,\mu)a_{\sigma\Delta\mu})=\lambda(\sigma,\mu)\Psi(\sigma\Delta\mu)a_{\sigma\Delta\mu}$, $\hat{\Psi}(a_\sigma)\hat{\Psi}(a_\mu)=\Psi(\sigma)a_\sigma\cdot\Psi(\mu)a_\mu = \lambda(\sigma,\mu)\Psi(\sigma)\Psi(\mu)\cdot a_{\sigma\Delta\mu}$ and, on the other direction, $\hat{\Psi}(a_\mu)\hat{\Psi}(a_\sigma)=\Psi(\mu)a_\mu\cdot\Psi(\sigma)a_\sigma = \lambda(\mu,\sigma)\Psi(\mu)\Psi(\sigma)\cdot a_{\sigma\Delta\mu}$. If $\{\sigma,\mu\}=1$, we have $\Psi(\sigma\Delta\mu)=\Psi(\sigma)\Psi(\mu)$ and then $\hat{\Psi}(a_\sigma a_\mu) = \hat{\Psi}(a_\sigma)\hat{\Psi}(a_\mu)$. If $\{\sigma,\mu\}=-1$, we get $\lambda(\mu,\sigma)=-\lambda(\sigma,\mu)$ and then, if $\Psi(\sigma\Delta\mu)=\Psi(\sigma)\Psi(\mu)$ we have $\hat{\Psi}(a_\sigma a_\mu)=\hat{\Psi}(a_\sigma)\hat{\Psi}(a_\mu)$, and if $\Psi(\sigma\Delta\mu)=-\Psi(\sigma)\Psi(\mu)$ we have $\hat{\Psi}(a_\sigma a_\mu)=\hat{\Psi}(a_\mu)\hat{\Psi}(a_\sigma)$. Therefore, $\hat{\Psi}$ is a half-automorphism in $H_\mathcal{B}(L)$.\\
		Conversely, assume $\varphi\in H_\mathcal{B}(L)$ and define, $\Psi(\emptyset)=1$ and $\Psi(\sigma)=\epsilon_\sigma$ where $\varphi(a_\sigma)=\epsilon_\sigma a_\sigma$ for $\sigma\in P_n\backslash\{\emptyset\}$. We shall prove that $\Psi\in R_\mathcal{B}$. By definition, $\Psi(\emptyset)=1$ and since $\epsilon_{\{j\}}=1$, for all $j\in I_n$ we get $\Psi(\{j\})=1$, for all $j\in I_n$. If $\{\sigma,\mu\}=1$, then $\varphi(a_\sigma a_\mu)=\varphi(a_\sigma)\varphi(a_\mu)$ and $\epsilon_{\sigma\Delta\mu}=\epsilon_\sigma\epsilon_\mu$. Then $\Psi(\sigma\Delta\mu)=\Psi(\sigma)\Psi(\mu)$ if $\{\sigma,\mu\}=1$. It is clear that $\varphi=\hat\Psi$. 
	\end{proof}
	The set $W$ of all mappings from $P_n$ to $\{1,-1\}$ is an $\mathbb{F}_2$ vector space of dimension $2^{2^n}$. The set $R_\mathcal{B}$ is an $\mathbb{F}_2$ subspace of $W$. In fact, for $\Psi_1$ and $\Psi_2\in R_\mathcal{B}$, $\Psi_1(\emptyset)\Psi_2(\emptyset)=1$, $\Psi_1(\{j\})\Psi_2(\{j\})=1$ for every $j\in I_n$ and $\Psi_1(\sigma\Delta\mu)\Psi_2(\sigma\Delta\mu)=\Psi_1(\sigma)\Psi_2(\sigma)\cdot\Psi_1(\mu)\Psi_2(\mu)$, whenever $\{\sigma,\mu\}=1$. 
	\begin{lemma}
		The group $H_\mathcal{B}(L)$ is an elementary abelian group.
	\end{lemma}
	\begin{proof}
		It suffices to prove that $\Psi\in R\mapsto\hat{\Psi}\in H(L)$ is an isomorphism of groups. If $\Psi$ and $\Phi\in R_\mathcal{B}$, for every $a_\sigma\in L$,
		\begin{equation}\nonumber
			(\hat{\Psi}\circ\hat{\Phi})(a_\sigma)=\hat{\Psi}(\hat{\Phi}(a_\sigma)) = \hat{\Psi}(\Phi(\sigma)a_\sigma) = \Phi(\sigma)\hat{\Psi}(a_\sigma) = \Phi(\sigma)\Psi(\sigma)a_\sigma = \Psi(\sigma)\Phi(\sigma)a_\sigma.
		\end{equation}
		Then, we conclude that $\Psi\cdot\Phi\in R_\mathcal{B}\mapsto\hat{\Psi}\circ\hat{\Phi}\in H_\mathcal{B}(L)$. Finally, we have $\hat{\Psi} = \hat{\Phi}$ if and only if $\hat{\Psi}(a_\sigma) = \hat{\Phi}(a_\sigma)$ for every $\sigma\in P_n$, which is the same as $\Psi(\sigma)=\Phi(\sigma)$ and then $\Psi=\Phi$. Then $H_\mathcal{B}(L)$ is group isomorphic to $R_\mathcal{B}$, thus $H_\mathcal{B}(L)$ is an elementary abelian group.
	\end{proof}
	\begin{remark}\label{Rm:intersection}
		Since the only automorphism of $L$ that fixes all the elements of a basis is the identity of $L$, it is easy to see that, for every basis $\mathcal{B}$, $\aut(L)\cap H_\mathcal{B}(L)=\{Id\}$.
	\end{remark}
	
	\begin{theorem}\label{Th:Structure}
		Let $L$ be a code loop with fixed basis $\mathcal{B}=\{a_1,\ldots,a_n\}$. Consider $H_\mathcal{B}(L)$ the set of all half-automorphisms $\varphi$ of $L$ such that $\varphi(a_i)=a_i$, for every $i=1,\ldots,n$. Then, for every half-automorphism $f$ of $L$, there are unique $h\in\aut(L)$ and $\varphi\in H_\mathcal{B}(L)$ such that $f=  h\varphi$. In particular, $\half(L)=\aut(L)\cdot H_\mathcal{B}(L)$.
	\end{theorem}
	\begin{proof}
		Let $f$ be a half-automorphism of $L$. If $f(a_i)=b_i$, then, by Corollary \ref{Cr:same-char-vec}, $\{b_1,\ldots,b_n\}$ is a basis of $L$ corresponding the same characteristic vector as $\mathcal{B}$. By Lemma 3.2 \cite{BP}, there is a unique $g\in\aut(L)$ such that $g(b_i)=a_i$, for every $i=1,\ldots,n$. Then $\varphi=gf$ is a half-automorphism in $H_\mathcal{B}(L)$. Setting $h=g^{-1}$, we get $f=h\varphi$. Now, if $f=h\varphi = h_1\varphi_1$, where $h,h_1\in\aut(L)$ and $\varphi,\varphi_1\in H_\mathcal{B}(L)$, then $\varphi\varphi_1^{-1}=h^{-1}h_1\in\aut(L)\cap H_\mathcal{B}(L)$ and then $\varphi = \varphi_1$ and $h=h_1$, by Remark \ref{Rm:intersection}.
	\end{proof}
	\begin{corollary}
		The group of half-automorphisms of a code loop $L$ is a product of the automorphism group of $L$ by an elementary abelian $2-$group.
	\end{corollary}
	\begin{corollary}
		The structure of the group $H_\mathcal{B}(L)$ does not depend on $\mathcal{B}$.
	\end{corollary}
	\begin{proof}
		For two bases $\mathcal{B}$ and $\mathcal{B}_1$ of $L$, we have $|\aut(L)|.|H_\mathcal{B}(L)|=|\aut(L)|.|H_{\mathcal{B}_1}(L)|$ and then, $|H_\mathcal{B}(L)|=|H_{\mathcal{B}_1}(L)|$. Since both groups are elementary abelian $2-$groups, we have $H_\mathcal{B}(L)$ isomorphic to $H_{\mathcal{B}_1}(L)$.
	\end{proof}
	\begin{remark}
		The trivial half-automorphism $J: x\mapsto x^{-1}$ belongs to the group $H_\mathcal{B}(L)$ if and only if every element of $B$ has order $2$. In the case $J\notin H_\mathcal{B}(L)$, there is a $\varphi\in H_\mathcal{B}(L)$, anti-automorphism, such that $J\varphi^{-1}$ is an automorphism of $L$. In particular, $L$ possesses a nontrivial half-automorphism if and only if $|H_\mathcal{B}(L)|>2$.
	\end{remark}
	% \textcolor{red}{Rose, veja se esse texto abaixo fica bom para você depois digitar o algoritmo do Grichkov}\\
	% \textcolor{red}{Aqui começa o algoritmo - Se a gente for deixar o algoritmo como seção, acho que esse texto ficaria melhor como parte inicial de seção}
	%%%%%%%%%%%%%%%%%%%%%

	\subsection{Computing the Subspace \( R_\mathcal{B} \)} \label{algorithm}
	By Proposition \ref{Pr:RBHB} and Theorem \ref{Th:Structure}, in order to compute the number of all half-automorphisms of a given code loop of rank $n$ given a fixed basis $\mathcal{B}$, it suffices to compute the subspace $R_\mathcal{B}$ and it can be done by solving a linear system obtained by the following algorithm: Given a code loop $L$ of rank \( n \), with a fixed basis $\mathcal{B}$, and a characteristic vector 
	\[
	\lambda = (\lambda_1,\ldots,\lambda_n,\lambda_{1,2},\ldots,\lambda_{n-1,n},\lambda_{1,2,3},\ldots,\lambda_{n-2,n-1,n}),
	\]
	where $\lambda_i$, $\lambda_{i,j}$, $\lambda_{i,j,k}\in\{0,1\}$, with $i<j<k$.
	
	\begin{enumerate}[label={\bfseries Step \arabic*.}, leftmargin=*]
		\item Define the product structure:
		\begin{itemize}
			\item For each \( \sigma \subseteq P_n \), define \( a_\sigma = a_{\sigma \setminus \{i\}} \cdot a_i \), where \( i = \max \sigma \).
			\item Compute the product: \( a_\sigma \cdot a_\mu = \lambda(\sigma, \mu) \cdot a_{\sigma \triangle \mu} \), where $a_\sigma \in L$.
			\item Compute the commutator: \( \{\sigma, \mu\}\).
		\end{itemize}
		
		\item Define an order on the set \( P_n = \{\sigma \subseteq I_n\} \) by:
		\begin{itemize}
			\item \( \sigma < \mu \) if \( |\sigma| < |\mu| \);
			\item if \( |\sigma| = |\mu| \), then \( \sigma < \mu \) if \( \max(\sigma) < \max(\mu) \);
			\item if \( |\sigma| = |\mu| \) and \( \max(\sigma) = \max(\mu) \), then compare \( \sigma \setminus \{i\} \) and \( \mu \setminus \{i\} \), with \( i = \max(\sigma) \).
		\end{itemize}
		
		\item Define a linear system over \( \mathbb{F}_2 = \{1,-1\} \), with multiplicative notation:
		\[
		\varepsilon_{\sigma \triangle \mu} = \varepsilon_\sigma \cdot \varepsilon_\mu \quad \text{whenever } \{\sigma, \mu\} = 1.
		\]
		
		\item Initialize with the values:
		\[
		\varepsilon_{\emptyset} = \varepsilon_1 = \dots = \varepsilon_n = 1.
		\]
		
		% \item For each \( \sigma \in P_n \) with \( |\sigma| \geq 2 \), define \( \varepsilon_\sigma \) inductively, assuming \( \varepsilon_\mu \) is known or undefined for all \( \mu < \sigma \). Let \( \{\lambda \in P_n \mid \lambda < \sigma \} = \{\sigma_1 < \dots < \sigma_m\} \). For \( k = 1,\ldots, m \), if \( \sigma_k \triangle \sigma > \sigma \), skip to \( \sigma_{k+1} \). If this holds for all \( \sigma_k \), proceed to the next \( \sigma \).
		
		\item Compute the sets:
		% $$
		% \widehat{A}_\sigma = \{ (\sigma_i, \sigma_j) \mid \sigma_i \triangle \sigma_j = \sigma,\ \sigma_i, \sigma_j < \sigma \}\,\,\mbox{and}\,\,   A_\sigma = \{ (\sigma_i, \sigma_j) \in \widehat{A}_\sigma \mid \{\sigma_i, \sigma_j\} = 1 \}.
		% $$
		\begin{align*}
			\widehat{A}_\sigma 
			&= \{ (\sigma_i, \sigma_j) \mid \sigma_i \triangle \sigma_j = \sigma,\ \sigma_i, \sigma_j < \sigma \}, \\
			A_\sigma 
			&= \{ (\sigma_i, \sigma_j) \in \widehat{A}_\sigma \mid \{ \sigma_i, \sigma_j \} = 1 \}.
		\end{align*}
		
		\item For $\sigma$ with $|\sigma|\geq2$, determine \( \varepsilon_\sigma \) inductively as the following:
		
		\begin{itemize}
			\item If \( A_\sigma = \emptyset \), declare \( \varepsilon_\sigma \) as an undefined variable.
			\item  If $(\sigma_i,\sigma_j)\in \mathcal{A}_\sigma$ set \( \varepsilon_\sigma = \varepsilon_{\sigma_i} \cdot \varepsilon_{\sigma_j} \). In case $(\sigma_r,\sigma_s)\in\mathcal{A}_\sigma$, we have a relation \( \varepsilon_{\sigma_i} \cdot \varepsilon_{\sigma_j} = \varepsilon_\sigma = \varepsilon_{\sigma_r} \cdot \varepsilon_{\sigma_s}\), and then we can delete repeated variables.
		\end{itemize}
		\item The system always admits a solution (taking $\varepsilon_\sigma = 1$, for all $\sigma\in P_n$), and the subspace $R_\mathcal{B}$ is its solution subspace.
	\end{enumerate}

	\section{Half-automorphisms of code loops of rank $3$ and $4$}\label{Sc:Rank3and4}
	
	In this section, we apply the results of the Section \ref{Sc:HA} in order to have a concrete look on the half-automorphism group of all code loops of rank $3$ and $4$. Given a code loop $L$, by Theorem \ref{Th:Structure},  $|\half(L)|=|H_\mathcal{B}(L)|\cdot|\aut(L)|$, for some basis $\mathcal{B}$. In order to compute $|\aut(L)|$, it is defined the  \emph{outer automorphism group of} $L$ as the quotient group $\out(L)=\aut(L)/N$, where $N=\{\sigma\in\aut(L);\sigma(x)=\pm x,\ \mbox{for all}\ x\in L\}$, \cite{Rose, GP18}. 
	Observe that, for all automorphisms $f$ and $g$ in $\aut(L)$, we have $fN=gN$ if and only if $g(x)=f(\pm x)=\pm f(x)$ for every $x\in L$, therefore, if $L$ is a code loop of rank $n$, $|\aut(L)|=2^n|\out(L)|$. Also, in Proposition 3.7 of \cite{GP18}, the structure of $\out(L)$ is described, where $L$ is a code loop of rank $4$.
	
	\subsection{Code loops of rank $3$} Let $L$ be a code loop of rank $3$ with a basis $\mathcal{B}=\{a_1,a_2,a_3\}$, and a characteristic vector chosen as Proposition 3.2 of \cite{GP18}. For $n=3$, we have the set $P_3 = \{\emptyset, \{1\},\{2\},\{3\},\{1,2\},\{1,3\},\{2,3\},\{1,2,3\}\}$ and we are searching for all possible functions $\Psi$ in the set $R_\mathcal{B}$ and then we can construct the corresponding half-automorphism $\hat{\Psi}$ as Proposition \ref{Pr:RBHB}.   We denote the function $\Psi$ as a list of its images and by the definition of $R_\mathcal{B}$ we have $\Psi=[1,1,1,1,\Psi(\{1,2\}),\Psi(\{1,3\}),\Psi(\{2,3\}),\Psi(\{1,2,3\})]$ and the other images depend on the commutator of the elements of $\mathcal{B}$.
	
	Let $\Psi \in R_\mathcal{B}$. Let us determine all possible values of $\Psi$ by applying the steps given by Algorithm established in Subsection \ref{algorithm}. For each $\sigma \in P_3$ such that $|\sigma| \geq 2$, we need to calculate the sets defined in Step $5$. Thus, we obtain: 
	
	\begin{enumerate}
		\item $\widehat{A}_{\{i,j\}} = \{ (\{i\}, \{j\})\}$, if $\{i,j\} = \{1,2\}$ or $\{i,j\} = \{1,3\};$
		\item $\widehat{A}_{\{2,3\}} = \{ (\{2\}, \{3\}), (\{1,2\}, \{1,3\})\};$
		\item $\widehat{A}_{\{1,2,3\}} = \{ (\{1\}, \{2,3\}), (\{2\}, \{1,3\}), (\{3\}, \{1,2\})\}.$
	\end{enumerate}
	
	Therefore, 
	
	\begin{enumerate}
		\item $\Psi(\{i,j\})=\Psi(\{i\}\triangle \{j\})$, if $\{i,j\} = \{1,2\}$ or $\{i,j\} = \{1,3\};$
		\item $\Psi(\{2,3\})=\Psi(\{2\}\triangle \{3\}) = \Psi(\{1,2\}\triangle \{1,3\})$;
		\item $\Psi(\{1,2,3\})=\Psi(\{1\}\triangle \{2,3\}) = \Psi(\{2\}\triangle \{1,3\}) = \Psi(\{3\}\triangle \{1,2\})$.
	\end{enumerate}
	
	To determine these images, we must compute the sets $A_\sigma$ by identifying which commutators $\{\sigma_1, \sigma_2\} = [a_{\sigma_1}, a_{\sigma_2}]$  are equal to $1$, for each pair $(\sigma_1,\sigma_2) \in \widehat{A}_\sigma$. We now compute these sets for each code loop of rank $3$.

	Let  $L=C_1^3$ be the code loop with characteristic vector $(1,1,1,1,1,1)$. Since $[a_1,a_2]=[a_1,a_3]=[a_2,a_3]=[a_1,a_2a_3]=[a_2,a_1a_3]=[a_3,a_1a_2]=-1$, we get $A_\sigma = \emptyset$, for all $\sigma \in P_3$, with $|\sigma| \geq 2$ and we obtain $4$ undefined variables, denoted by $x,y,z,t$. Then $\Psi = [1,1,1,1,x,y,z,t]$ 
	with $x,y,z,t\in\{1,-1\}$, and $|R_\mathcal{B}| = 16$. 
	According to Proposition \ref{Pr:RBHB}, all 16 possible $\hat{\Psi}$ are half-automorphisms of $L$. These results can also be observed in \cite{BP}, where it is pointed out that $|\half(L)|=21504$ and $|\aut(L)|=1344$. 
	
	Similarly, we can analyze the code loop $L=C_3^3$ with characteristic vector $(0,0,0,1,1,1)$. We also obtain $\Psi = [1,1,1,1,x,y,z,t]$ with $x,y,z,t\in\{1,-1\}$, and $|R_\mathcal{B}| = |H(L)| = 16$. As an example, the function $\Psi = [1,1,1,1,1,1,1,-1]$ corresponds to a half-automorphism $\hat{\Psi}$ such that $\hat{\Psi}(a_1a_2\cdot a_3) = -a_1a_2\cdot a_3$, $\hat{\Psi}(-a_1a_2\cdot a_3) = a_1a_2\cdot a_3$, and $\hat{\Psi}(b) = b$ otherwise. Since $a_1a_2\cdot a_3$ is an element of order $2$, $\hat{\Psi}$ is not an elementary mapping in the sense of Definition 2.10 of \cite{BP}. According to GAP computations, $\hat{\Psi}$ is a nontrivial half-automorphism that is not the product of an elementary mapping and an automorphism of $L$.
	We must point out that in Theorem 4.3 of \cite{BP} the authors proved that every half-automorphisms of $L$ is a composition of an automorphism of $L$ and an element of the group generated by all elementary mappings. This proof is incorrect and we realized that the mistake is to suppose that the set of generators $\{x,y,z\}$, in the notation of the proof of Theorem 4.3, could be associated to the characteristic vector $(000111)$.
	Actually, by Theorem \ref{Th:Structure} and since $|R_\mathcal{B}|=16$, we get $|\half(L)|=3072$, because $|\aut(L)|=192$.
	
	Now, let $L=C_2^3$ be the code loop with characteristic vector $(0,0,0,0,0,0)$. Since $[a_1,a_2]=[a_1,a_3]=[a_2,a_3]=1$ and $[a_1,a_2a_3]=[a_2,a_1a_3]=[a_3,a_1a_2]=-1$, we get $A_\sigma \neq \emptyset$, for all $\sigma \in \{\{1,2\},\{1,3\},\{2,3\}\} $, and $A_{\{1,2,3\}} = \emptyset$. In this case, $\Psi(\{1,2\}) = \Psi(\{1,3\}) = \Psi(\{2,3\}) = 1$ and, there exists only one undefined variable $\varepsilon_{\{1,2,3\}} = x$. Thus, $\Psi = [1,1,1,1,1,1,1,x]$ with $x\in\{1,-1\}$, and $|R_\mathcal{B}| = 2$. According to Proposition \ref{Pr:RBHB}, these 2 possible $\hat{\Psi}$ are half-automorphisms of $L$. This results can also be observed in \cite{BP}, where it is stated that $|\half(L)|=384$ and $|\aut(L)|=192$.
	
	For the code loops $L=C_4^3$ and 
	$L=C_5^3$, with characteristic vectors $(1,1,0,0,0,0)$ and  $(1,0,0,0,0,0)$, respectively, the analyses and results are the same as for $L=C_2^3$. We get  $\Psi = [1,1,1,1,1,1,1,x]$ with $x\in\{1,-1\}$, and $|R_\mathcal{B}| = |H(L)| = 2$. Moreover, these results can also be observed in \cite{BP}, where it is shown that 
	$|\half(C_4^3)|=128$, $|\aut(C_4^3)|=64$, $|\half(C_5^3)|=96$ and $|\aut(C_5^3)|=48$.

	%%%%%
	
	\subsection{Code loops of rank $4$} Let $L$ ba a code loop of rank $4$. We claim that $L$ possesses a code subloop of rank $3$, $L'$  with no nontrivial half-automorphism. In fact, one can consider $L'$ as a nonabelian subgroup and then, by Scott's result $|H(L')|=2$. Now, let $L'=\langle a_{i_1},a_{i_2},a_{i_3}\rangle$ and consider $\mathcal{B}=\{a_{i_1},a_{i_2},a_{i_3},a_{i_4}\}$ a basis of $L$. Consider the map $\theta:H_\mathcal{B}(L)\to H(L')$ defined by $\theta(\varphi)=\varphi|_{L'}$. Since $\varphi$ fixes all elements of $\mathcal{B}$, then $\theta(\varphi)$ is indeed a half-automorphism of $L'$ and, also $\theta$ is a homomorphism of groups. Therefore, $|H_\mathcal{B}(L)|=|H(L')||ker\theta|=2|\ker\theta|$, where $\ker\theta$ contains all $\varphi\in H_\mathcal{B}(L)$ such that $\varphi(a_\sigma)=a_\sigma$, for all $\sigma\in P_4$ with $i_4\notin\sigma$. Let us compute an example: Let $C^4_9$ be the code loop of rank $4$ with characteristic vector $(0100001000)$ of \cite{GP18}. This is the loop CodeLoop(32,11) in GAP Library. This loop has a basis $\mathcal{B}=\{a_1,a_2,a_3,a_4\}$ and $a_4$ is a nuclear element of $L$. We have that $L'=\langle a_1,a_3,a_4\rangle$ is a nonabelian group, then $|H(L')|=2$. If $\varphi$ is an element of $ker\theta$, then $\varphi=\hat{\Psi}$ and $\Psi(\sigma)=1$ for every $\sigma\in P_4$ with $2\notin\sigma$. As in the case of $n=3$, let us write $\Psi$ as $\Psi=[\Psi(\emptyset),\Psi(\{1\}),\Psi(\{2\}),\Psi(\{3\}),\Psi(\{4\}),\Psi(\{1,2\}),\Psi(\{1,3\}),\Psi(\{1,4\}),\Psi(\{2,3\}),\break\Psi(\{2,4\}),\Psi(\{3,4\}),\Psi(\{1,2,3\}),\Psi(\{1,2,4\}),\Psi(\{1,3,4\}),\Psi(\{2,3,4\}),\Psi(\{1,2,3,4\})]$.\break Then, $\Psi=[1,1,1,1,1,\Psi(\{1,2\}),1,1,\Psi(\{2,3\}),\Psi(\{2,4\}),1,\Psi(\{1,2,3\}),\Psi(\{1,2,4\}),1,\break\Psi(\{2,3,4\}),\Psi(\{1,2,3,4\})]$. Since $[a_1,a_2]=[a_2,a_3]=[a_2,a_4]=[a_2,a_1a_4]=[a_2,a_3a_4]=[a_1,a_2a_3\cdot a_4]=1$, we get $\Psi(\{1,2\})=\Psi(\{2,3\})=\Psi(\{2,4\})=\Psi(\{1,2,4\})=\Psi(\{2,3,4\})\break=\Psi(\{1,2,3,4\})=1$. Finally, since $[a_1a_2\cdot a_3,a_1a_4]=[a_1a_2\cdot a_3,a_1][a_1a_2\cdot a_3,a_4]=(-1)(-1)=1$, we get $\Psi(\{2,3,4\})=\Psi(\{1,2,3\})\Psi(\{1,4\})$. We get $ker\theta=\{\Psi\}$ where $\Psi=[1,1,\ldots,1]$, then $|H(L)|=2$ and $L$ has no nontrivial half-automorphisms. In the following table, we summarize the size of $|\half(L)|$ for all code loops of rank $4$, according to the characteristic vectors displayed in Theorem \ref{Th:G-Rose-posto4}. 
	
	\begin{table}[h!]
		\centering
		\renewcommand{\arraystretch}{1.2}
		\begin{tabular}{|c||c|c|c|c|c|}
			\hline 
			$i$ & $L = \cl(32,i)$ & $|\out(L)|$ & $|\aut(L)|$ & $|H(L)|$ & $|\half(L)|$ \\ \hline \hline
			1  & $C^{4}_2$   & 192   & 3072    & 2    & 6144    \\ \hline
			2  & $C^{4}_3$   & 192   & 3072    & 16    & 49152   \\ \hline
			3  & $C^{4}_{14}$   & 48     & 768     & 2     & 1536    \\ \hline
			4  & $C^{4}_7$   & 24    & 384     & 2     & 768     \\ \hline
			5  & $C^{4}_5$   & 48    & 768     & 2     & 1536    \\ \hline
			6  & $C^{4}_8$   & 24    & 384     & 2     & 768     \\ \hline
			7  & $C^{4}_{10}$ & 12   & 192     & 2     & 384     \\ \hline
			8  & $C^{4}_{16}$ & 8    & 128     & 2     & 256     \\ \hline
			9  & $C^{4}_{15}$ & 48   & 768     & 2     & 1536    \\ \hline
			10 & $C^{4}_1$   & 1344  & 21504   & 16  & 344064  \\ \hline
			11 & $C^{4}_9$   & 8     & 128     & 2     & 256     \\ \hline
			12 & $C^{4}_{12}$ & 8    & 128     & 2     & 256    \\ \hline
			13 & $C^{4}_4$   & 64     & 1024     & 2   & 2048     \\ \hline
			14 & $C^{4}_{13}$ & 24    & 384     & 2     & 768     \\ \hline
			15 & $C^{4}_6$   & 168    & 2688     & 16     & 43008     \\ \hline
			16 & $C^{4}_{11}$ & 12   & 192     & 2     & 384     \\ \hline
		\end{tabular}
		%\caption{}
	\end{table}

	\newpage
	\section{Elementary mappings}\label{Sc:elementary-map}
	In this section, we point out some facts about the so called elementary mappings. Despite those mappings play a very important role when a code loop of rank $3$ has a  nontrivial half-automorphism, we show that for code loops of rank at least $4$, no elementary mapping can be a half-automorphism. Recall the Definition 2.10 of \cite{BP}, for a diassociative loop $L$ and an element $c$ in $L$ the \emph{elementary mapping} $\tau_c:L\to L$ is defined by $\tau_c(x)=x^{-1}$ if $x\in\{c,c^{-1}\}$ and $\tau_c(x)=x$ otherwise.
	\begin{lemma}\label{lm:1}
		Let $L$ be a nonassociative code loop and let $x\in L$ be such that $x^2 = -1$. The following statements are equivalent:
		\begin{itemize}
			\item[(i)] $\tau_x$ is a half-automorphism;
			\item[(ii)] For any $y,z\in L$:
			\begin{itemize}
				\item[(ii.a)] $[x,y] = 1$ implies $y\in\subsp{x}$;
				\item[(ii.b)] $(x,y,z) = 1$ implies either $y\in\subsp{x,z}$ or $z\in\subsp{x,y}$
			\end{itemize}
			\item[(iii)] If $\lambda = (\lambda_1,\ldots,\lambda_n,\lambda_{12},\ldots,\lambda_{n-1,n},\lambda_{123},\ldots,\lambda_{n-2,n-1,n})$ is the characteristic vector related to the generating set $\{v_1,\ldots,v_n\}$ with $v_1=x$, then $\lambda_{1j}=\lambda_{1jk}=1$, for $1<j<k$.
		\end{itemize}
	\end{lemma}
	\begin{proof}
		Assume (i). For (ii.a), let $y\in L$ with $[x,y]=1$ and $y\notin\subsp{x}$. Then $y^{-1}x\notin\subsp{x}$ and we get $[y,y^{-1}x]=[y,y^{-1}][y,x](y^{-1},x,y)=1$. Since $\tau_x$ is a half-automorphism and $[y,y^{-1}x]=1$ we get $\tau_x(y\cdot y^{-1}x) = \tau_x(y)\tau_x(y^{-1}x)$ and then  $$-x = \tau_x(x) = \tau_x(y\cdot y^{-1}x)=\tau_x(y)\tau_x(y^{-1}x)=y\cdot y^{-1}x = x, $$ which is a contradiction.\\
		For (ii.b), let $y,z\in L$ with $(x,y,z) = 1$ and suppose that $y\notin\subsp{x,z}$. 
		%In particular, $y\notin\subsp{x}$ and, by item (i), $[x,y]=-1$.
		If $z\in\subsp{x}$, then $z\in\subsp{x,y}$. Assume, towards a contradiction, that $z\notin\subsp{x}$ and, by item (ii.a), $[x,z] = -1$.
		By a simple calculation, $\tau_x(xz) = xz$, $\tau_x(x)\tau_x(z)=-xz$ and $\tau_x(z)\tau_x(x) = -zx$ and since $[x,z] = -1$, we get $\tau_x(xz) = \tau_x(z)\tau_x(x)$. By Moufang's Theorem and Scott's result, we conclude that $\tau_x$ restricts to an anti-homomorphism in $\subsp{x,y,z}$. We have $yz = \tau_x(yz) = \tau_x(z)\tau_x(y) = zy$ which implies $[y,z] = 1$ and then, $[z,xy] = [z,x][z,y](z,x,y) = -1$. Then, $$z\cdot xy = \tau_x(z\cdot xy) = \tau_x(xy)\tau_x(z) = xy\cdot z = - z\cdot xy,$$ a contradiction, and then (ii) holds. If we assume (ii), then (iii) holds by the definition of characteristic vector. Now, suppose (iii) and let $a$ and $b$ be elements of $L$. First, suppose that $a,b\in\langle x\rangle=\{1,-1,x,-x\}$. Since every central element of $L$ has order $2$, $\tau_x(ab)=ab=\tau_x(a)\tau_x(b)$, if $a,b\in Z(L)$. If $a\in Z(L)$ and $b\notin Z(L)$, then we have $\tau_x(ab)=\tau_x(\pm b)=\pm\tau_x( b)=\tau_x(a)\tau_x(b)$. If $a,b\notin Z(L)$, then $a,b\in Z(L)$ and $\tau_x(ab)=\tau_x(\pm1)=\pm1=ab=\tau_x(a)\tau_x(b)$. Now, suppose that $a,b\notin\langle x\rangle$. We first note that in this case one can consider a basis of $L$ containing either $\{x,a\}$ or $\{x,b\}$ then, by (iii), $[x,a]=[x,b]=-1$. If $ab\notin\langle x\rangle$, $\tau_x(ab)=ab=\tau_x(a)\tau_x(b)$; if $ab=\pm1$, $\tau_x(ab)=\tau_x(\pm1)=\pm1=ab=\tau_x(a)\tau_x(b)$. For the case where $ab=x$, then $1=[b,b]=[b,a^{-1}x]=[b,a^{-1}][b,x](b,a^{-1},x)$ and then $[a,b]=[b,a^{-1}]=-1$, since $(b,a^{-1},x)=(b,a^{-1},ab)$. Then $\tau_x(ab)=-ab=ba=\tau_x(b)\tau_x(a)$. A similar argument can be used for the case $ab=-x$. The only remaining case is the case that $a\in\{x,-x\}$ and $b\notin\langle x\rangle$. In this case, we have $ab\notin\langle x\rangle$, $[x,b]=-1$ and then $[a,b]=-1$. Therefore, $\tau_x(ab)=ab=-ba=
		\tau_x(b)\tau_x(a)$ and $\tau_x$ is a half-automorphism.
	\end{proof}
	\begin{corollary}\label{Cr:n=4}
		Let $L$ be a code loop of rank $4$ and let $x$ be an element of $L$ with $x^2=-1$. Then, the elementary mapping $\tau_x$ is not a half-automorphism. 
	\end{corollary}
	\begin{proof}
		Suppose $x\in L$, with $x^2=-1$ and, by Lemma 3.4 of \cite{GP18}, consider $\{x,y,z,t\}$ a basis of $L$ in which $(x,y,z)=-1$ and $t$ is a nuclear element of $L$. Then, if $\tau_x$ is a half-automorphism of $L$, we would have $(x,y,t)=1$, $t\notin\subsp{x,y}$ and $y\notin\subsp{x,t}$, which is a contradiction.
	\end{proof}
	\begin{corollary}\label{Cr:nucleus}
		Let $L$ be a nonassociative code loop and let $x\in L$ be with $x^2=-1$. If $\tau_x$ is a half-automorphism, then $x\notin N(L)$. 
	\end{corollary}
	\begin{proof}
		First, since $x^2=-1$ there is a basis of $L$ containing $x$. By Lemma \ref{lm:1}, if $\tau_x$ is a half-automorphism of $L$, then for any $y,z\in L$ with $(x,y,z)=1$ we have either $y\in\subsp{x,z}$ or $z\in\subsp{x,y}$. Then $x$ must not lie in $N(L)$ because $x$ can not associate with the others generators of $L$. 
	\end{proof}
	\begin{corollary}\label{Cr:n>=4}
		Let $L$ be a nonassociative code loop of rank $n\ge 4$ and let $x\in L$. If $x^2=-1$ then $\tau_x$ is not a half-automorphism.
	\end{corollary}
	\begin{proof}
		For $n=4$ we have Corollary \ref{Cr:n=4}. In the other cases, if $x^2=-1$ and $\tau_x$ is a half-automorphism, we have $x\notin N(L)$ and then there is a nonassociative subloop $L'$ of rank $4$ containing $x$. By the definition of $\tau_x$, it restricts to a half-automorphism of $L'$ which contradicts Corollary \ref{Cr:n=4}.
	\end{proof}
	Non-identical elementary mappings given by elements only have the chance to be a half-automorphism when the rank of the code loop is $3$. Then we aim at generalizing this concept for code loops with greater ranks.
	\begin{condition}\label{Cd:subgroupA}
		Suppose that $L$ is a nonassociative code loop with a maximal abelian
		subgroup $A$ satisfying the following properties:
		\begin{itemize}
			\item[(i)] For all $a$ and $b$ in $A$ and $x$ in $L$, we have $(a,b,x)=1$;
			\item[(ii)] For all $a\in A$ and $x\notin A$, we have $[a,x]=a^2$;
			% \textcolor{blue}{
				% \item[(iii)] For all $a\in A$ with $a^2=-1$ and $x,y\notin A$ we have either $y\in\subsp{a,x}$ or $x\in\subsp{a,y}$ whenever $(a,x,y) = 1$.}
			\item[(iii)] There exists $a\in A$ such that $a^2=-1$;
		\end{itemize}
	\end{condition}
	\begin{definition}
		For a subloop $A$ in a diassociative loop $L$, the generalized elementary mapping of the subloop $A$, $\tau_A:L\to L$ is defined by 
		$$\tau_A(x)=\begin{cases}
			x^{-1},&\mbox{if}\quad x\in A,\\
			x, &\mbox{otherwise}.
		\end{cases}$$
	\end{definition}
	\begin{remark}
		In the case where $L$ is a nonassociative code loop of rank $3$, the maximal abelian subgroups $A$ which satisfy item (i) and (iv) of Condition \ref{Cd:subgroupA} are exactly the cyclic subgroups of $L$. In fact, suppose that $A$ has two generators $a$ and $b$, with $a^2=-1$ and let $L=\langle x,y,z\rangle$. Since $|A|=|\langle a\rangle\langle b\rangle|=\dfrac{|\langle a\rangle||\langle b\rangle|}{|\langle a\rangle\cap\langle b\rangle|}$, we get $|A|=8$ no matter the order of $b$. Since $A$ is not all $L$, we can assume without loss of generality that $x\notin A$. Then $|\langle A,x\rangle|=16$ and $L=\langle a,b,x\rangle$. By item (i), $(a,b,x)=1$, and by Moufang's Theorem $L$ is associative, which is a contradiction.
	\end{remark}
	
	\begin{lemma}\label{Lm:tau_A}
		Let $L$ be a code loop with maximal abelian subgroup $A$ satisfying Condition \ref{Cd:subgroupA}. Then the generalized elementary mapping $\tau_A$ is a nonidentical half-automorphism.  
	\end{lemma}
	\begin{proof}
		By item (iv), we have that $\tau_A$ is not the identical map.
		Let $x, y\in L$. Clearly, $\tau_A(xy)=\tau_A(x)\tau_A(y)$ if either $x,y\in A$ or $x,y,xy\notin A$. If $x,y\notin A$ and $xy=a\in A$, we get $y=x^{-1}a$ and $y^{-1} = a^{-1}x$. We have $\tau_A(xy)=\tau_A(a)=a^{-1}$. If $a^2=1$, then $\tau_A(x)\tau_A(y)=xy=a=a^{-1}$. If $a^2=-1$, by item (ii) we get $$\tau_A(y)\tau_A(x)=yx=x^{-1}a\cdot x = x^{-1}\cdot ax = -x^{-1}\cdot xa= -a = a^{-1}.$$ Finally, consider $a\in A$ and $x\notin A$. If $a^2=1$, we get $\tau_A(ax)=ax=\tau_A(a)\tau_A(x)$ and, similarly $\tau_A(xa)=\tau_A(x)\tau_A(a)$. If $a^2=-1$, then $\tau_A(ax)=ax=-xa=x(-a)=\tau_A(x)\tau_A(a)$ and, similarly, $\tau_A(xa)=\tau_A(a)\tau_A(x)$. Therefore, $\tau_A$ is a half-automorphism.
	\end{proof}
	Using GAP, we search within all $16$ code loops of rank $4$ the maximal abelian subgroups satisfying Condition \ref{Cd:subgroupA}. As a result, only $4$ of them have such subgroup: $C_2^4, C_3^4, C_1^4$ and $C_4^4$. For $C_2^4$ and $C_4^4$ there exists only one generalized elementary mapping which is an anti-automorphism. On the other hand, $C_3^4$ and $C_1^4$ have, respectively, $3$ and $7$ generalized elementary mappings that are nontrivial half-automorphisms. The loop $C_6^4$ has nontrivial half-automorphism but does not have generalized elementary mapping, and then we notice that there exists nontrivial half-automorphism which is not related to generalized elementary mapping.
	
	We finish this paper with:
	\begin{problem}
		What conditions must the subgroup $A$ satisfy for the generalized elementary mapping $\tau_A$ to be a nontrivial half-automorphism?
	\end{problem}
	
	Based on Proposition \ref{Pr:code-loops-isomorphic}, we suggest the following:
	
	\begin{problem}
		Let $f:L\to Q$ be a half-isomorphism of two Moufang loops. Is $L$ isomorphic to $Q$?
	\end{problem}
	\subsection*{Acknowledgments}
	
	The computational calculations in this work have been made by using the LOOPS package \cite{NV1} for GAP \cite{gap}.  The authors thank the National Council for Scientific and Technological Development CNPq (grant 406932/2023-9). Grishkov was supported by FAPESP (grant 2024/14914-9), CNPq (grant 307593/2023-1), and in accordance with the state task of the IM SB RAS, project FWNF-2022-003. The authors thank Giliard Souza dos Anjos for productive conversation about the topic.
	
	% {\bf Declarations}

	% {\bf Ethical standard}  Not applicable in this manuscript.
	
	% {\bf Conflict of interest} The authors declare no competing interests.
	
	%%%%%%%%%%%%%%%%%%%%%%%%%%%
	%%%%%%%%%%%%%%%%%%%%%%%%%%%
	%%%%%%%%%%%%%%%%%%%%%%%%%%%
	%%%%%%%%%%%%%%%%%%%%%%%%%%%

\end{document}